\documentclass[11pt]{amsart}

\usepackage{bbm,bm}
\usepackage{amssymb}

\usepackage{subfigure}

\usepackage{color}
\usepackage{pstricks,pst-node,pst-poly}

	\definecolor{myred}{rgb}{.5,.1,.1}
	\definecolor{mygreen}{rgb}{.1,.5,.1}
	\definecolor{myblue}{rgb}{.1,.1,.5}
	\definecolor{mycyan}{cmyk}{.7,0,0,0}
	\definecolor{mymagenta}{cmyk}{0,.7,0,0}

\usepackage[bookmarks=false,colorlinks,linkcolor=myblue,citecolor=mygreen]{hyperref}

\usepackage{subfigure}
\usepackage[width=.85\textwidth,font=small,labelfont=sc]{caption}

\PII{}
\copyrightinfo{}{}
\date{19 June 2010}

\newpsobject{showgrid}{psgrid}{%
    subgriddiv=1,griddots=10,gridlabels=6pt}
\newif\ifhrule\hrulefalse

\newlength{\xSW} \newlength{\xNE} \newlength{\xO}%
\newlength{\ySW} \newlength{\yNE} \newlength{\yO}%
\newcommand{\Wordbox}[6]{%
\setlength{\xSW}{#2\psxunit} %
\setlength{\xNE}{#4\psxunit} %
	\setlength{\xO}{.5\xSW}%
	\addtolength{\xO}{.5\xNE}%
\setlength{\ySW}{#3\psyunit} %
\setlength{\yNE}{#5\psyunit} %
	\setlength{\yO}{.5\ySW}%
	\addtolength{\yO}{.5\yNE}%
	\psframe[#1](#2,#3)(#4,#5)%
	\rput[c](\xO,\yO){#6}
	}
\def\wordbox(#1,#2)(#3,#4)#5{%
	\Wordbox{dimen=middle,fillstyle=solid,fillcolor=white,linewidth=0.01,linecolor=black}{#1}{#2}{#3}{#4}{#5}%
	}

\theoremstyle{plain}
\newtheorem{theorem}{Theorem}

\newtheorem{corollary}[theorem]{Corollary}
\newtheorem{lemma}[theorem]{Lemma}

\theoremstyle{definition}
\newtheorem{definition}[theorem]{Definition}
\newtheorem{example}[theorem]{Example} 

\theoremstyle{remark}
\newtheorem*{remark}{Remark} 
\newtheorem*{remarks}{Remarks} 

\newcommand{\demph}[1]{\textcolor{myblue}{\it #1}}

\newcommand{\bN}{\mathbb{N}}

\newcommand{\bP}{\mathbb{P}}
\newcommand{\bQ}{\mathbb{Q}}

\newcommand{\frakS}{\mathfrak{S}}

\newcommand{\A}{\mathbf{A}}
\newcommand{\B}{\mathbf{B}}
\newcommand{\C}{\mathbf{C}}

\newcommand{\nsym}{\textsl{NSym}}

\newcommand{\ncsym}{\textsl{NCSym}}

\newcommand{\Comps}{\bm\Gamma}
\newcommand{\Parts}{\bm\Pi}
\newcommand{\Atoms}{\bm{\dot\Pi}}

\newcommand{\length}[1]{\ell\left(#1\right)}
\def\mrg{{\textcolor{mymagenta}{|}}}

\def\cmrg{\bm|}

\def\restrict{{}_{\!}\rfloor}

\newcommand{\nshuff}{\mathrel{\raise1pt\hbox{$\scriptscriptstyle\cup{\mskip-4mu}\cup$}}}
\newcommand{\qshuff}{\nshuff}
\newcommand{\lqshuff}{\mathrel{\tilde\nshuff}}
\newcommand{\djcup}{\mathrel{\ensuremath{\mathaccent\cdot\cup}}}
\newcommand{\shift}[2]{{#1^{\scriptstyle+#2}}}
\newcommand{\std}[1]{{#1}^{\scriptstyle\downarrow}}

\def\ot{\otimes}
\def\id{\mathrm{id}}

\author{Aaron Lauve}
\address[Lauve]{
	Department of Mathematics\\
        Texas A\&M University\\
        College Station, TX\, 77843 \\
        (USA)
        }
\email{lauve@math.tamu.edu}
\urladdr{http://www.math.tamu.edu/$\small\sim$lauve}
\thanks{Lauve partially supported by NSA grant \#H98230-10-1-0362.}

\author{Mitja Mastnak}
\address[Mastnak]{
	Department of Mathematics \& Computing Science\\
        Saint Mary's University\\
        Halifax, Nova Scotia\, B3H 3C3 \\
        (CANADA) \ 27109
	}
\email{mmastnak@cs.smu.ca}
\urladdr{http://www.cs.smu.ca/$\small\sim$mmastnak}
\thanks{Mastnak partially supported by an NSERC Discovery Grant.}

\title[The Primitives and Antipode of $\ncsym$]{The Primitives and Antipode in the {H}opf Algebra of Symmetric Functions in Noncommuting Variables}


\begin{document}

\begin{abstract}
We identify a collection of primitive elements generating the Hopf algebra $\ncsym$ of symmetric functions in noncommuting variables and give a combinatorial formula for the antipode. 
\end{abstract}

\maketitle
\section{Introduction}\label{sec: intro}

The Hopf algebra $\ncsym$ of symmetric functions in noncommuting variables was introduced in \cite{BRRZ:2008}, following the work of Wolf, Rosas, and Sagan \cite{Wol:1936,RosSag:2006}. We refer the reader to these sources for details on its realization as formal sums of monomials invariant under $\frakS_n$ (for all $n>0$). For our goals, it suffices to describe $\ncsym$ abstractly in terms of generators and relations, which we do in Section \ref{sec: ncsym}. 

The Hopf algebra $\ncsym$ is cocommutative (by definition) and freely generated as an algebra (essentially the main result of Wolf). Following \cite{BRRZ:2008}, we choose the \emph{atomic set partitions} $\Atoms$ as a free generating set (see Section \ref{sec: ncsym}). The Cartier-Milnor-Moore theorem then guarantees that $\ncsym$ is isomorphic to $\mathfrak U(\mathfrak L(\Atoms))$, the universal enveloping algebra of the free Lie algebra generated by $\Atoms$. This note grew out of an attempt to realize this isomorphism explicitly. 

In Section \ref{sec: prelims}, we record some useful combinatorial machinery. Section \ref{sec: ncsym} contains a precise definition of the Hopf algebra $\ncsym$ as well as the statements and proofs of our main results. In Section \ref{sec: summary}, we comment on: (i) an important connection between $\ncsym$ and supercharacter theory; and (ii) key structural features of our proofs that will be further developed in \cite{LauMas:1}.

\subsection*{Acknowledgements} We thank Marcelo Aguiar and Nantel Bergeron for useful conversations and for pointing us to the related work of Patras and Reutenauer. 

\section{Combinatorial Preliminaries}\label{sec: prelims}
We record some useful shorthand for manipulating set partitions and set compositions. Throughout, we let $\bN$ and $\bP$ denote the set of nonnegative integers and positive integers, respectively. Also, given $n\in\bP$, we let $[n]$ denote the subset $\{1,2,\dotsc,n\}$. 

\subsection{Set partitions}\label{sec: set parts}
Fix $X\subseteq \bP$ and let $\A = \{A_1,\dotsc, A_r\}$ be a set of subsets of $X$. We say that $\A$ is a \demph{set partition} of $X$, written $\A \vdash X$, if and only if $A_1\cup\dotsb\cup A_r=X$, $A_i\neq\emptyset$ ($\forall i$), and $A_i\cap A_j=\emptyset$ ($\forall i\neq j$). 
We order the parts in increasing order of their minimum elements. 
The \demph{weight} $|\A|$ of $\A$ is the cardinality of $X$ and the \demph{length} $\length{\A}$ of $\A$ is its number of parts ($r$). 
In what follows, we lighten the heavy notation for set partitions, writing, e.g., the set partition $\{\{1,3\},\{2,8\},\{4\}\}$ as $13.28.4$. We write $\Parts(X)$ for the set partitions of $X$ and $\Parts(n)$ when $X=[n]$.

Given any $A\subseteq \bN$ and $k\in \bN$, we write $\shift{A}{k}$ for the set 
\[
	\shift{A}{k}:=\{a+k \mid a\in A\}.
\]
By extension, for any set partition $\A=\{A_1,A_2,\ldots,A_r\}$ we set 
$
	\shift{\A}{k}:=\{\shift{A_1}{k},\shift{A_2}{k},$ $\ldots,\shift{A_r}{k}\}.
$
The operator $\shift{(\hbox{--})}{k}$ has a complement $\std{(\hbox{--})}$ called the \demph{standardization} operator. It maps set partitions $\A$ of any cardinality $n$ subset $X\subseteq \bP$ to set partitions of $[n]$, by defining $\std{\A}$ as the pullback of $\A$ along the unique increasing bijection from $[n]$ to $X$. For example, $\std{(18.4)} = 13.2$ and $\std{(18.4.67)}=15.2.34$. 
Given set partitions $\B \vdash [m]$ and $\C \vdash [n]$, we let $\B \mrg \C$ denote the set partition $\B \cup \shift{\C}{m}$ of $[m+n]$. 

\begin{definition}\label{def: atomic} 
A set partition $\A=\{A_1,A_2,\ldots,A_r\}$ of $[n]$ is \demph{atomic} (``connected'' in \cite{HivNovThi:2008}) if there {does not} exist a subset $\B\subseteq \A$ and an integer $m<n$ such that $\B \vdash [m]$. Conversely, $\A$ is not atomic if there are set partitions $\B$ of $[n']$ and $\C$ of $[n'']$ splitting $\A$ in two: $\A = \B \mrg \C$. 
\end{definition}

For example, the partition $17.235.4.68$ is atomic, while $12.346.57.8$ is not. The maximal splitting of the latter would be $12 \mrg 124.35 \mrg 1$.
We denote the atomic set partitions by $\Atoms$.

If $\A$ is a set partition with $r$ parts, and $K\subseteq[r]$, we write $A_K$ to denote the sub partition $\{A_k \mid k\in K\}\subseteq \A$. For example, if $\A=17.235.4.68$, then $\A_{\{1,3,4\}} = 17.4.68$.

\subsection{Set compositions}\label{sec: set comps}
Fix $K\subseteq \bP$ and let $\gamma = (\gamma_1,\dotsc, \gamma_s)$ be a sequence of subsets of $K$. We say that $\gamma$ is a \demph{set composition} of $K$, written $\gamma \vDash K$, if and only if $\gamma_1\cup\dotsb\cup \gamma_s=K$, $\gamma_i\neq\emptyset$ ($\forall i$), and $\gamma_i\cap \gamma_j=\emptyset$ ($\forall i\neq j$). 
The \demph{weight} $|\gamma|$ and \demph{length} $\length{\gamma}$ are defined as for set partitions.
We use ``$\cmrg$'' in place of ``.'' in our shorthand for set compositions, e.g., the set composition $(\{3,8\},\{1,2\},\{4\})$ is abbreviated as $38\cmrg12\cmrg4$. We write $\Comps(K)$ for the set partitions of $K$ and $\Comps(r)$ when $K=[r]$.

If $\gamma$ is a set composition of $X$ and $K\subseteq X$, we write $\gamma\restrict_K$ for the induced set composition of $K$. For example, if $\gamma = 38\cmrg 12\cmrg 4$ and $K = \{3,4,8\}$, then $\gamma\restrict_K = 38\cmrg 4$. Similarly, $\gamma\restrict_{\{1,3\}} = 3\cmrg 1$.
Note that $\gamma\restrict_{\{1,3\}}$ is not the same as $\gamma_{\{1,3\}}$. Following the notation introduced for set partitions, we let $\gamma_{\{1,3\}}$ denote the subsequence $38\cmrg 4$ of $\gamma$. 

Given two set compositions $\gamma,\rho \vDash K$, we say that $\gamma$ \demph{refines} $\rho$, written $\gamma \succ \rho$, if each block of $\rho$ is the union of a contiguous string of blocks of $\gamma$. For example, $2\cmrg4\cmrg3\cmrg17\cmrg 9 \succ 234\cmrg179 \succ 123479$.

\subsection{Set compositions as functions on set partitions}\label{sec: functions}
Let $\A$ be a set partition of $X$ with $r$ parts and suppose $\gamma=(\gamma_1,\dotsc,\gamma_s)$ is a set composition of $K\subseteq[r]$. We define a new set partition $\gamma[\A]$ as follows:
\begin{gather}\label{eq: gamma of A}
	\gamma[\A] := \std{\A_{\gamma_1}} \mrg \std{\A_{\gamma_2}} \mrg \dotsb \mrg \std{\A_{\gamma_{s}}} \,.
\end{gather}
See Figure \ref{fig: gamma of A} for several examples. 
\begin{figure}[hbt]
\[
\begin{array}{r|lll}
 & 13.29.458.7 & 13.28.456.7 & 15.28.346.7  \\[0.4ex]
\hline
13\cmrg2\  & 12.345.67 & 12.345.67 & 14.235.67 \\
2\cmrg34\  & 12.346.5 & 12.345.6 & 12.345.6 \\
1\cmrg234\  & 12.38.457.6 & 12.38.456.7 & 12.38.457.6  \\
\end{array}
\]
\caption{Set partitions $\gamma[\A]$ for several examples of $\gamma$ and $\A$.}
\label{fig: gamma of A}
\end{figure}
%

\section{Structure of $\ncsym$}\label{sec: ncsym}

Let $\ncsym=\bigoplus_{n\geq0} \ncsym_n$ denote the graded $\bQ$ vector space whose $n$th graded piece consists of formal sums of set partitions $\A \in \Parts(n)$. We give $\ncsym$ the structure of graded Hopf algebra as follows. The algebra structure is given by
\begin{gather}\label{eq: ncsym product} 
	\A \bm\cdot \B = \A \mrg \B \quad\hbox{and}\quad 1_{\ncsym} = \bm\emptyset \,,
\end{gather}
where $\bm\emptyset$ is the unique set partition of the empty set. The coalgebra structure is given by
\begin{gather}\label{eq: ncsym coproduct} 
	\Delta(\A) = \sum_{K\djcup L = [r]} \std{(\A_K)} \otimes \std{(\A_L)} \quad\hbox{and}\quad \varepsilon(\A) = 0 \hbox{\ for all\ }\A\neq\bm\emptyset \,.
\end{gather}
Here, $\A$ is a set partition with $r$ parts, and $K\djcup L = [r]$ is understood as an ordered disjoint union, i.e., $K\cup L = [r], K\cap L = \emptyset,$ and $K\djcup L \neq L \djcup K$. 

In terms of the symmetric function interpretation of $\ncsym$, the formulas above correspond to working in the basis of \demph{power sum symmetric functions}. The compatibility of product and coproduct is proven in \cite{BRRZ:2008}, the product formula is proven in \cite{BerZab:2009}, and the coproduct formula is proven in \cite{BHRZ:2006}. A combinatorially meaningful formula for the antipode in $\ncsym$ has been missing until now (Theorem \ref{thm: ncsym antipode}). 

Evidently, $\ncsym$ is cocommutative and freely generated by $\Atoms$. The Cartier-Milnor-Moore theorem guarantees algebraically independent primitive elements $p(\A)$ of $\ncsym$ associated to each $\A \in\mathcal P$. We find them below (Theorem \ref{thm: ncsym primitives}). 

\subsection{Primitive generators of $\ncsym$}\label{sec: ncsym primitives}
We aim to prove the following result.

\begin{theorem}\label{thm: ncsym primitives} 
Let $\Comps'(r)$ denote the set compositions $\gamma$ of $[r]$ with $1\in\gamma_1$. If $\A$ is a set partition with $r$ parts, then
\begin{gather}\label{eq: ncsym primitives} 
p(\A) := \sum_{\gamma\in\Comps'(r)} (-1)^{\length{\gamma}-1} \gamma[\A]
\end{gather}
is a nonzero primitive if $\A\in\Atoms$ and zero otherwise. 
\end{theorem}

\begin{remark} An arbitrary choice was made in \eqref{eq: ncsym primitives}, namely demanding that the distinguished element $1$ belongs to the first block of $\gamma$. Different choices of distinguished elements (and distinguished blocks) give rise to different sets of primitives. Summing over these choices produces a projection operator onto the space of all primitives. 
This operator can be understood in the context of the work Patras and Reutenauer on Lie idempotents \cite{PatReu:1999,PatReu:2002a}. 
Fisher \cite{Fis:thesis} uses similar idempotents to work out primitive formulas for several Hopf monoids in species. His formulas, which were discovered independently, also give rise to \eqref{eq: ncsym primitives}.
\end{remark}

To prove Theorem \ref{thm: ncsym primitives}, we need a lemma about \emph{left quasi-shuffles.} Let $\mathbf{P}=\langle \mathrm{fin}(2^\bP) \rangle$ denote the free monoid generated by the finite subsets of $\bP$ and let $\bQ\mathbf{P}$ denote the corresponding free algebra. We use $\mrg$ to separate letters in the words $w\in\mathbf{P}$. Given a word $u=u_1\mrg\dotsb\mrg u_k$ on $k$ letters and an index $i\leq k$, we write $u_{[i]}$ for the prefix $u_1\mrg\dotsb\mrg u_i$ and $u^{[i]}$ for the suffix $u_{i+1}\mrg\dotsb\mrg u_k$.

We say that words $u=u_1\mrg\dotsb\mrg u_k$ and $v=v_1\mrg\dotsb\mrg v_l$ in $\mathbf{P}$ are \emph{disjoint} if $\bigl(\bigcup_i u_i\bigr) \cap \bigl(\bigcup_j v_j\bigr) = \emptyset$. (We identify the set compositions $\Comps(K)$ with the words $w\in\mathbf{P}$ that further satisfy $w_i\cap w_j = \emptyset$ for $i\neq j$ and $\bigcup_j w_j = K$.) 

The \demph{quasi-shuffle} $u\qshuff v$ of two disjoint words $u,v\in \mathbf{P}$ is defined recursively as follows:
\begin{itemize}
\item $u \qshuff \emptyset = u$ and $\emptyset \qshuff v = v$;
\item if $u=a\mrg u'$ and $v=b\mrg v'$, then 
\[
	u \qshuff v = \left\{a\mrg w : w\in u'\qshuff v \right\} \cup
		\left\{ab\mrg w : w\in u'\qshuff v' \right\} \cup
		\left\{b\mrg w : w\in u\qshuff v' \right\} .
\]
\end{itemize}
Here $\emptyset$ denotes the unique set composition of the emptyset and ``$ab$'' denotes the union $a\cup b$, a single letter in $\mathbf{P}$. 
The quasi-shuffle governs the formula for the product of two monomial symmetric functions in noncommuting variables. We need a subset of these that we call the \demph{left quasi-shuffles:} 
\[
	u \lqshuff v = \left\{a\mrg w : w\in u'\qshuff v \right\} \cup
		\left\{ab\mrg w : w\in u'\qshuff v' \right\} 
\]
for nonempty disjoint words $u=a\mrg u'$ and $v=b\mrg v'$ in $\mathbf{P}$. (Note that the recursive definition involves $\qshuff$, not $\lqshuff$.) 

\begin{example} Consider the set compositions $1\mrg3$ and $24$. We have
\[
	1\mrg 3 \lqshuff 24 = \bigl\{1\mrg 3 \mrg 24, \, 1\mrg 234, 1\mrg 24\mrg 3, \, 124\mrg 3\bigr\} \,.
\]
Note that for each left quasi-shuffle $\gamma$ above, $\gamma\restrict_{\{1,3\}} = 1\mrg3$ and $\gamma\restrict_{\{2,4\}} = 24$.
\end{example}

\begin{lemma}\label{thm: mcsym primitives} Let $\Comps'(r)$ denote the set compositions $\gamma$ of $[r]$ with $1\in\gamma_1$. Given any partition $K\djcup L = [r]$ with $1\in K\subsetneq[r]$, we have
\[
	\sum_{\gamma\in\Comps'(r)} (-1)^{\length{\gamma}} \gamma\restrict_K \otimes \gamma\restrict_L = 0 \,,
\]
as an element of $\bQ\mathbf{P}\otimes \bQ\mathbf{P}$. 
\end{lemma}

\begin{proof} Consider a term $u \otimes v$ in the sum. (The hypothesis $K\neq [r]$ guarantees that $v\neq\emptyset$.) The compositions $\gamma$ that satisfy $\gamma\restrict_K = u$ and $\gamma\restrict_L = v$ are precisely the left quasi-shuffles $u \lqshuff v$. We now establish a bijection between those of even length and those of odd length. This will complete the proof. 

A left quasi-shuffle $w\in u\lqshuff v$ falls into one of two types according to whether or not the first letter of $v$ appears as a letter in $w$: after beginning with some (possibly empty) initial prefix of $u$, say up to the $i$th letter $u_i$, the rest of the word $w$ looks like either $u_{i+1}\mrg v_1\mrg w'$ or $u_{i+1}v_1\mrg w'$ for $w'$ in $u^{[i+1]}\qshuff v^{[1]}$. Here, again, $u_{i+1}v_1$ denotes $u_{i+1}\cup v_1$. The indicated pairing $w \mapsto \phi(w)$ ($u_{i+1}\mrg v_1 \mapsto u_{i+1}v_1$) decreases the number of letters by one. Thus $w$ and $\phi(w)$ contribute opposite signs to the coefficient of $u\otimes v$. 
\end{proof}

\begin{proof}[Proof of Theorem \ref{thm: ncsym primitives}]
Since $\Delta$ is cocommutative, we need only consider the terms $\std{\gamma[\A]_K} \otimes \std{\gamma[\A]_L}$ from each $\Delta(\gamma[\A])$ satisfying $1\in K$. Notice that ${\gamma[\A]_K} = {\bigl(\gamma\restrict_{K'}\bigr)[\A]}$ for some $K'\subseteq[r]$ (with $1\in K'$ if $1\in\gamma_1$). Thus
\begin{gather*}
\sum_{\gamma\in\Comps'(r)} (-1)^{\length{\gamma}} 
	\sum_{\substack{K\djcup L = [r] \\ 1\in K}} 
	\std{\gamma[\A]_K} \otimes \std{\gamma[\A]_L}
\intertext{%
is the same sum as}
\sum_{\substack{K\djcup L = [r] \\ 1\in K}} 
	\sum_{\gamma\in\Comps'(r)} (-1)^{\length{\gamma}} 
	\std{\bigl(\gamma\restrict_K\bigr)[\A]} \otimes \std{\bigl(\gamma\restrict_L\bigr)[\A]} \,,
\intertext{%
or even}
\Biggl(\sum_{\substack{K\djcup L = [r] \\ 1\in K}} 
	\sum_{\gamma\in\Comps'(r)} (-1)^{\length{\gamma}} 
	\gamma\restrict_K \otimes \gamma\restrict_L \Biggr)[\A\otimes \A]
	\,.
\end{gather*}
Lemma \ref{thm: mcsym primitives} reduces this expression to ${(0)[\A\otimes\A]}=0$ when $K\neq[r]$. When $K=[r]$, the expression is precisely $p(\A)\otimes 1$. Conclude that $p(\A)$ is a primitive element when $\A$ is atomic. 

To show that $p(\A)$ is nonzero when $\A$ is atomic, we merely remark that $\gamma[\A]$ has at least as many atoms as $\gamma$ has parts. It is only when $\A\in\Atoms$ and $\gamma=[r]$ that a single atom is obtained: $\gamma[\A]$ contributes $\A$ to the sum $p(\A)$ only when $\gamma=[r]$. Thus $p(\A)\neq 0$.

We now show that $p(\A)=0$ when $\A$ is not atomic. Suppose $\A=\A'\mrg\A''$ and put $\length{\A} = r, \length{\A'} = r'$. (Note that $r'<r$ by assumption.) Divide the set compositions $\gamma\in \Comps'(r)$ into two types according to the following dichotomy. If $\gamma_i$ is the first letter, from the left, satisfying $\gamma_i \cap \{r'+1,\dotsc,r\} \neq \emptyset$, then either $\gamma_i\cap\{1,\dotsc,r'\}$ is empty or it is not. The pairing $\gamma\mapsto \phi(\gamma)$ ($\gamma_{i-1}\mrg\gamma_i \mapsto \gamma_{i-1}\gamma_i$) decreases the number of letters by one. However, the difference is not visible at the level of functions on $\A$. That is, $\gamma[\A] = \phi(\gamma)[\A]$. This completes the proof.
\end{proof}

\begin{corollary}The set $\{p(\A) \mid \A\in\Atoms\}$ comprises irredundant (algebraically independent) generators of the Lie algebra of primitive elements of $\ncsym$.
\end{corollary}

\begin{proof}
First we show algebraic independence of $\{p(\A) \mid \A\in\Atoms\}$. Let $<$ be a total order on the set $\Atoms$ satisfying $\A < \A'$ when $|\A|>|\A'|$. If we extend $<$ to a lexicographic ordering of $\Parts$ in the usual manner, then the leading (minimum) term of $p(\A)$ is $\A$. (As remarked in the proof above, the only term in $p(\A)$ with one atom is $\A$.) Conclude that any polynomial in the $p(\A)$s ($\A\in \Atoms$) has the same leading term as the corresponding polynomial in the $\A$s. Hence the $p(\A)$s freely generate $\ncsym$.

Turning to the Lie algebra of primitive elements in $\ncsym$, we recall the construction of Hall polynomials.
Given an ordered alphabet $X$ and a word $w=x_1\dotsc x_t$ over $X$, we say that $w$ is a \demph{Lyndon word} if $w$ is lexicographically smaller than all its proper suffixes $x_i\dotsb x_t$ ($i>1$). A classical result has that all Lyndon words $w$ have a unique proper decomposition $w=uv$ with $v$ Lyndon of maximum length. See \cite{Reu:1993}. We define the Hall polynomial $[[w]]$ by forming the Lie bracket $[u,v]$ at successively smaller Lyndon factorizations. For example, if $w=aabb$, then the Lyndon factorization of $w$ is $(a,abb)$; next, $abb$ is further factored as $(ab,b)$ and $ab$ is factored as $(a,b)$. The resulting Hall polynomial is $[[w]] = [a,[[a,b],b]]$. 

If $w$ is Lyndon, then the leading term of $[[w]]$ is $w$. A consequence is the classical result that the Hall polynomials $\{[[w]] \mid w\hbox{ is Lyndon}\}$ form a basis of the free Lie algebra generated by $X$. Turning to $\ncsym$, the first paragraph of the proof shows that we may replace the alphabet $\Atoms$ with the alphabet $\{p(\A) \mid \A \in \Atoms\}$. Conclude that the Hall polynomials
$[[p(\A')p(\A'')\dotsb p(\A^{(t)})]]$, with $\A'\mrg \A''\mrg \dotsb\mrg \A^{(t)}$ Lyndon in the atoms $\A^{(i)}$, form a basis of the Lie algebra of primitive elements in $\ncsym$.
\end{proof}

\subsection{The antipode of $\ncsym$}\label{sec: ncsym antipode}
We aim to prove the following result.

\begin{theorem}\label{thm: ncsym antipode} 
Suppose $\A$ is a set partition with $r$ parts and let $\Comps(r)$ denote the set compositions of $[r]$. The antipode $S$ of $\ncsym$ acts on $\A$ by 
\begin{gather}\label{eq: ncsym antipode} 
S({\A}) = \sum_{\gamma\in\Comps(r)} (-1)^{\length{\gamma}} \gamma[\A] \,.
\end{gather}
\end{theorem}

\begin{remarks} {\it 1.} This description of the antipode typically contains many cancellations. For example, it says that 
$S(12.3) = -(12.3) + (12.3) + (1.23)$. 
There may even be cancellations when $\A$ is atomic. We invite the reader to compute $S(14.2.3)$, which {\it a priori} has 13 terms, but in fact has only nine.
On the other hand, \eqref{eq: ncsym antipode} is irredundant for many atomic $\A$.
\smallskip

{\it 2.} In case $\A$ is not atomic, we can do better. Suppose $\A=\A'\mrg\A''\mrg\dotsb\mrg\A^{(t)}$ is a splitting of $\A$ into atomic pieces, and let $\length{\A^{(i)}} = r_i$ with $\sum_i r_i = r$. Finally, let $\Comps(\overleftarrow r)$ denote all refinements of the set composition
\[
	\overleftarrow r = \bigl(\{r{-}r_t{+}1,\dotsc,r\},\dotsc,\{r_1{+}1,\dotsc,r_1{+}r_2\},\{1,\dotsc,r_1\}\bigr) .
\]
Then 
\begin{gather}\label{eq: alternate antipode}
	S(\A) = \sum_{\gamma\in\Comps(\overleftarrow r)} (-1)^{\length{\gamma}} \gamma[\A] \,. 
\end{gather}
This follows immediately from the fact that $S$ is an algebra antimorphism and \eqref{eq: ncsym antipode} holds on atomic set partitions. Using this formula, we may express $S(13.2.4)$ using three terms, $S(13.2.4) = (1.24.3) - (1.23.4) - (1.2.34)$. Using \eqref{eq: ncsym antipode} would have required us to write down 13 terms, 10 of which would cancel. 
The relationship between \eqref{eq: alternate antipode} and Theorem 14.31 of \cite{AguMah:1} will be explored in \cite{LauMas:1}.
\end{remarks}

\begin{proof}[Proof of Theorem \ref{thm: ncsym antipode}]
Since the antipode is guaranteed to exist in graded connected bialgebras, we need only check that \eqref{eq: ncsym antipode} provides a left convolution inverse of $\id$. Writing $m$ for multiplication in $\ncsym$, we have
\begin{align*}
m(S\ot\id)\Delta(\A) &= m(S\ot\id)\left(\sum_{K\djcup L=[r]} \std{\A_K} \ot \std{\A_L}\right) 
= m(S\ot\id)\left(\sum_{K\djcup L=[r]} K[\A]\ot L[\A]\right)  \\
\intertext{(viewing $K$ and $L$ as set compositions with one part)}
 &= m \sum_{K\djcup L=[r]} \left(\sum_{\gamma\in\Comps_{|K|}} (-1)^{\length{\gamma}}\gamma[K[\A]]\right)\ot L[\A] \\
&= \left(\sum_{K\djcup L=[r]} \sum_{\gamma\in\Comps_{K}} (-1)^{\length{\gamma}}(\gamma\cmrg L)\right) \! [\A] \,.
\end{align*}
We show that this sum is the zero function on $\A$. We have
\begin{align*}
\sum_{K\djcup L=[r]} \sum_{\gamma\in\Comps_{K}} (-1)^{\length{\gamma}}(\gamma\cmrg L) &=
	\sum_{\gamma\in\Comps_{[r]}} (-1)^{\length{\gamma}}\gamma \ \  + \ 
	\sum_{\substack{K\djcup L=[r] \\ K\neq[r]}} \sum_{\gamma\in\Comps_{K}} (-1)^{\length{\gamma}}(\gamma\cmrg L) \\
&= 	\sum_{\emptyset\subsetneq L\subseteq[r]}\sum_{\substack{\gamma\in\Comps_{[r]}\\ \gamma=\gamma'\cmrg L}} (-1)^{\length{\gamma'\cmrg L}}(\gamma'\cmrg L) \ \ - 
	\sum_{\substack{K\djcup L=[r] \\ K\neq[r]}} \sum_{\gamma\in\Comps_{K}} (-1)^{\length{\gamma\cmrg L}}(\gamma\cmrg L) \\
&= 0 \,,
\end{align*}
as claimed. We conclude that $m(S\ot\id)\Delta(\A) = 0$ for $|\A|>0$, which completes the proof.
\end{proof}

\section{Summary Remarks}\label{sec: summary}

\subsection{Supercharacter theory}\label{sec: SC}
Let $\mathbb{F}_q$ be a finite field with $q$ elements. The classification problem for irreducible representations of the upper triangular groups $U_n(\mathbb{F}_q)$ is known to be of wild type. After the work of Andr\'e and Yan \cite{And:1995,Yan:1}, it seems a good first step at understanding the character theory of $U_n(\mathbb{F}_q)$ is to study its \emph{supercharacters} and \emph{superclasses}. Superclasses are formed by clumping together certain conjugacy classes of $U_n(\mathbb{F}_q)$. To each superclass is associated a supercharacter, which is the corresponding sum of characters. See \cite{DiaIsa:2008} for an excellent exposition by Diaconis and Isaacs.  

Following this approach to the problem, and mimicing the classical constructions in the representation ring of the symmetric groups, Thiem \cite{Thi:2010} gave the space of supercharacters $\mathcal{SC}_q$ a Hopf algebra structure, with product and coproduct coming from superinflation and restriction.\footnote{More precisely, the coproduct is not explicitly defined there, nor is the product / coproduct compatibility checked, but these were subsequently verified  by Thiem (private communication). The details of his Hopf algebra construction will appear in the report \cite{AIM}.} The result was something that bore a strong resemblance to $\ncsym$ for $q=2$. One of the main topics of the AIM workshop \cite{AIM} was to explore this connection and see what could be said for $q$ arbitrary. Quite a lot of progress was made in several directions. Relevant to the present discussion is that $\mathcal{SC}_2$ is indeed isomorphic to $\ncsym$ as Hopf algebras. 

The isomorphism above is straightforward, simply taking superclass functions $\kappa_\A$ to monomial symmetric functions $m_\A$. (See \cite{Thi:2010} and \cite{BRRZ:2008} for notation.)
In light of this, it would be very interesting to have analogs of \eqref{eq: ncsym primitives} and \eqref{eq: ncsym antipode} for the monomial basis. Some exciting progress was made in this direction on the last day of the workshop, but some details still need to be checked before any formal statement can be made \cite{Nantel}.

\subsection{Transfer of structure}\label{sec: transfer}
In Section \ref{sec: ncsym}, we used the idea of set compositions as functions on set partitions to formulate our results. As the proofs of these results indicate, this idea can be mined further. There is a graded connected cocommutative Hopf algebra structure on $\bQ\mathbf{P}$ which is freely generated as an algebra by $\mathrm{fin}(2^\bP)$. Formulas for primitives and the antipode in $\bQ\mathbf{P}$ mimic \eqref{eq: ncsym primitives} and \eqref{eq: ncsym antipode}. 

More is true. In fact, the $\bQ\mathbf{P}$ formulas engender \eqref{eq: ncsym primitives} and \eqref{eq: ncsym antipode} via a transfer of structure coming from a \emph{measuring} of Hopf algebras. We leave the details to \cite{LauMas:1}, where further examples of transfer of structure will be worked out. We quote a key theorem from that work. 

\begin{theorem}\label{thm: transfer}
Let $A, B$ be Hopf algebras and let $C$ be a coalgebra. Suppose $\theta
\colon  B\ot C\to A$ is a \emph{covering} (a surjective coalgebra map that \emph{measures} $B$ to $A$). Let $\iota \colon A\to B\ot C$ be any linear section of $\theta$, that is, $\theta\circ \iota = \id$.
Then the following hold.
\begin{enumerate}
\item\label{itm: primitive} If $p\in B$ is primitive, then for every $c\in C$ the element $\theta(p,c)\in A$ is primitive.
\item\label{itm: antipode} $S_A = \theta\circ (S_B\ot\id)\circ \iota$.
\end{enumerate}
\end{theorem}

Here, $\bQ\mathbf{P}$ covers $\ncsym$ by taking $C$ to be the free pointed coalgebra on set partitions; the covering $\theta(\gamma\otimes \A)$ is the evaluation $\gamma[\A]$ discussed
in Section \ref{sec: functions}. In \cite{LauMas:1}, we also discuss coverings by $\nsym$, where $\nsym$ is the Hopf algebra of noncommutative symmetric functions \cite{GKLLRT:1995} freely generated by symbols $H_n$ that comultiply as
\[
	\Delta(H_n) = \sum_{i+j=n} H_i \otimes H_j \,.
\]
Such coverings exist for every graded cocommutative Hopf algebra $A$ and are given by setting $C$ to be the underlying coalgebra of $A$ and $\theta$ to be the unique covering for which $\theta(H_n \otimes a) := \pi_n(a)$, where $\pi_n$ is the projection to the $n^\mathrm{th}$ graded component of $A$. 
Again, primitive and antipode formulas in $\nsym$ are transfered via Theorem \ref{thm: transfer}. Formula \eqref{itm: antipode} in the theorem simply recovers Takeuchi's formula for the antipode, whereas Formula \eqref{itm: primitive} is used to obtain Takeuchi-type formulas for primitives in $A$. In particular, we obtain generators for the space of primitives (analogous to Theorem \ref{thm: ncsym primitives}) as well as a projection onto the space of all primitives.

\bibliographystyle{abbrv}
\bibliography{bibl}

\end{document}